\documentclass[12pt]{article}
\usepackage[top=1in, bottom=1in, left=1in, right=1in]{geometry}
\usepackage{geometry}                		
\usepackage[utf8]{inputenc}
\usepackage{graphicx}
\usepackage[english]{babel}
\usepackage{amsmath}
\usepackage{amsfonts}
\usepackage{amsthm}
\usepackage{color}
\usepackage{float}
\usepackage{caption}
\usepackage{subcaption}
\usepackage[noend,algonl]{algorithm2e}
\SetKwProg{Fn}{Function}{}{end}

\newtheorem{lemma}{Lemma}
\newtheorem{proposition}{Proposition}
\newtheorem{theorem}{Theorem}
\newtheorem{corollary}{Corollary}
\newtheorem{remark}{Remark}

\title{Asymptotic expansions for the transmission eigenvalues of periodic scatterers of bounded support}
\author{Fioralba Cakoni\footnote{Department of Mathematics, Rutgers University, New Brunswick, NJ, USA (fc292@math.rutgers.edu)} \ \ and \ \ Shari Moskow\footnote{Department of Mathematics, Drexel University, Philadelphia, PA, USA (moskow@math.drexel.edu)}}
\date{}

\begin{document}

\maketitle
\begin{abstract} 
We consider the transmission eigenvalues  for a bounded scatterer with a periodically varying index of refraction, and derive the first order corrections to the limiting transmission eigenvalues. We assume the scatterer contrast to be of one sign, in which case the transmission eigenvalue problem can be written in terms of operators corresponding to a fourth order PDE with periodic coefficients. We perform two-scale asymptotics for this biharmonic type homogenization problem and show convergence estimates which   require a boundary corrector function, and this boundary corrector function appears in the formula for the transmission eigenvalues correction.  
\end{abstract}

\section{Introduction}
The transmission eigenvalue problem plays a fundamental role in scattering theory for inhomogeneous media. Transmission eigenvalues correspond to interrogating frequencies at which there exists an incident field that does not scatter by the medium. Despite its deceptively simple formulation—two elliptic PDEs in a bounded domain (one governing wave propagation in the scattering medium and the other in the background that occupies the support of the medium) that share the same Cauchy data on the boundary—the problem presents a remarkably intricate mathematical structure. In particular, it is a non-self-adjoint eigenvalue problem for a non-strongly elliptic operator, making the investigation of its spectral properties highly challenging. We refer the reader to \cite{CCH-book} for the significance of this problem in scattering phenomena and inverse scattering theory.

\noindent
More precisely, let $n(x)$ denote the refractive index of an inhomogeneous medium of bounded support, which is a perturbation of the homogeneous background medium with refractive index scaled to one. Define $\overline{D} := \operatorname{supp}(n-1)$. The transmission eigenvalue problem is then formulated as finding $v \neq 0$ and $w \neq 0$ satisfying
\begin{equation}\label{ten2}
\left\{
\begin{array}{rrcl}
\Delta w + k^2 n(x) w = 0 & \quad & \text{in } D, \\
\Delta v + k^2 v = 0 & \quad & \text{in } D,\\
w-v = 0 & \quad & \text{on } \partial D,\\
\displaystyle \frac{\partial w}{\partial \nu} - \frac{\partial v}{\partial \nu} = 0 & \quad & \text{on } \partial D,
\end{array}
\right.
\end{equation}
where $k > 0$ is the wave number, proportional to the interrogating frequency $\omega$. In this formulation, $u := w-v$ corresponds to the scattered field, which by virtue of the boundary conditions (provided $\partial D$ has some regularity), can be extended by zero into the exterior of $D$, whereas $v$ is the restriction to $D$ of the incoming incident wave. This formulation shows that a necessary condition for an incident wave to remain unscattered by the inhomogeneity $(D,n)$ is the existence of a nontrivial solution to \eqref{ten2}. The transmission eigenvalue problem is known to be non-self-adjoint \cite{CCH-book}, and complex transmission eigenvalues may occur, although only the real ones are physically relevant to non-scattering. Values of $k \in \mathbb{C}$ for which \eqref{ten2} admits nontrivial solutions $(w,v)$ are called transmission eigenvalues. Note that, it can be shown that real transmission eigenvalues can be determined from measured scattering data \cite{CCH-book, eigm1,eigm2}, hence they can be used to determine information about refractive index $n$ when solving the  inverse scattering problem. There is a vast literature on the spectral analysis of the transmission eigenvalue problem. The discreteness of the spectrum, completeness of eigenfunctions, and Weyl’s asymptotics have been established under various assumptions on $n-1$ in \cite{nguyenn, Ki, rob, vodev, vodev2}. In particular, if $n-1$ has a fixed sign uniformly in $D$, then there exists an infinite sequence of real transmission eigenvalues accumulating only at $+\infty$. 

\medskip
\noindent
In this work, we deal with the perturbation analysis of transmission eigenvalues when the inhomogeneous medium \((D,n)\), with \(D \subset \mathbb{R}^d\) for \(d=2,3\) bounded, is periodic and highly oscillating. More precisely, let \(\epsilon>0\) denote the characteristic size of the periodic unit cell, which is assumed to be small relative to the size of \(D\), and let \(Y=[0,\,1]^d\) be the rescaled unit cell. We assume that the refractive index is given by  
\[
n_\epsilon(x):=n(x/\epsilon) \in L^\infty(D),
\]  
with \(n\) periodic in \(y=x/\epsilon\) with period \(Y\).  
Our concern is how the real transmission eigenvalues perturb as \(\epsilon \to 0\). The homogenization theory for the corresponding direct scattering problem has been developed in \cite{CaMoPa,CaGuMoPa}, while the convergence of the real transmission eigenvalues to those of the homogenized medium was proven in \cite{CaHaHa}. The main goal of this paper is to provide an explicit first-order correction term in the asymptotic expansion of the real transmission eigenvalue. Since the correction to the homogenized transmission eigenvalue can be determined, the hope is that this correction term captures microstructural information of the periodic medium.   Such asymptotic analyses have been carried out for transmission eigenvalue problems in media containing small-size perturbations as the perturbation size tends to zero in \cite{CaHaMo,CaMoRo}. Our perturbation analysis is based on the work of \cite{Os}, extended to nonlinear eigenvalue problems in \cite{Mo,FuMo}. In particular, our approach makes use of the expression given in \cite{FuMo}.  

\medskip
\noindent
In this paper we formulate the transmission eigenvalue problem as a non-linear eigenvalue problem for a fourth-order partial differential operator given by (\ref{10}). Using two-scale asymptotics for the resolvent of the bi-Laplacian-type operator, we establish higher-order convergence estimates as $\epsilon \to 0$. While homogenization for periodic higher-order PDEs is known \cite{ Fr, high2, high,Pa}, these higher order estimates including boundary corrector functions appear to be new. We show first that the homogenized problem reduces to the eigenvalue problem with cell-averaged refractive index ${\overline n}$, recovering \cite{CaHaHa}. We then construct higher-order correctors: the first-order expansion includes only a boundary correction (consistent with the homogenization of direct scattering problems \cite{CaMoPa,CaGuMoPa}), while higher orders involve both bulk and boundary terms. The resolvent estimates yield asymptotics for simple real transmission eigenvalues, based on the formula in \cite[Theorem 3.2]{FuMo}. The leading correction involves the  boundary corrector integrated against the scaled eigenfunction. The boundary corrector satisfies a fourth-order boundary value problem with oscillatory coefficients, and it is highly desirable to understand the limiting behavior of  the boundary corrector as $\epsilon \to 0$. This  is a delicate issue in the theory of homogenization, and for the state of the art of this question for second-order PDEs is  discussed in \cite{GeMa1, GeMa2, MoVo1, SaVo}. Here, although we can prove that the boundary corrector is $L^2$-bounded with respect to $\epsilon$, and thus admits weak subsequential limit(s), we characterize its (non-zero) limit only in one dimension. We find that the first order corrections are not unique and depend on the interaction of the boundary of
the scatterer with the microstructure.  The two and three-dimensional cases, which are technically more involved, will be addressed in a subsequent paper.

\section{Description of problem}
We assume that the  bounded domain $D \subset \mathbb{R}^d$ has  $C^2$ boundary, and  as stated in the introduction we assume that  $n(y)$ is a bounded  periodic function of $y$ in the cell $Y= [0,1]^d $.  Let $H^2_0(D)$ denote the Sobolev space given by
$$H^2_0(D) := \left \{  u\in H^2(D) : u=\frac{\partial u}{\partial \nu}=0\text{ on $\partial D$} \right \}.$$ 
or, equivalently, the $H^2$ closure of $C^\infty_0(D)$ functions, equipped with the inner product 
$$ (u,v)_{H^2_0(D)} = (\Delta u , \Delta v)_{L^2(D)}.$$
Consider now the interior transmission eigenvalue problem (\ref{ten2}) formulated above  which has a periodic coefficient with period $\epsilon>0$, small compared to its support $D$.  Letting $\tau:=k^2$, we wish to find nontrivial  $w, v\in L^2(D)$ with $w-v\in H^2_0(D)$ satisfying  
\begin{alignat}{5}\label{tep}
\Delta w+\tau n(x/\epsilon) w &= 0 &\qquad & \mathrm{in} \ D \\
\Delta v+\tau v& =0 & &  \mathrm{in}\  D \\
w&= v  & &\mathrm{on} \ \partial D \\
\displaystyle \frac{\partial w}{\partial \nu}&=\displaystyle \frac{\partial v}{\partial \nu} & & \mathrm{on}\  \partial D.\label{tep1} 
\end{alignat}
Note that the boundary conditions are equivalent to saying that $w-v \in H^2_0(D)$. Here the eigenvalue parameter is $\tau:=k^2$. As already mentioned, the transmission eigenvalue problem is not self-adjoint, and in the spherical symmetric case it is known to have complex eigenvalues \cite[Chapter 6]{CCH-book}. Here we are concerned only with the real transmission eigenvalue, since they are the ones which can be measured from scattering data.  More precisely, here transmission eigenvalues refers to values of $\tau\in{\mathbb R}_+$ for which the problem (\ref{tep})-(\ref{tep1}) has a nontrivial solution.
In this work we limit ourselves to the case when $n(x/\epsilon)-1$ is of one sign, and for the calculations we assume that $n(x/\epsilon) -1 \geq c >0 $, with $c$ independent of $\epsilon$. In this case, an infinite number of real transmission eigenvalues are known to exist \cite{CaGiHa}. We are interested in the behavior of these transmission eigenvalues as the period size $\epsilon$ approaches zero. It is known from the work \cite{CaHaHa} that the real transmission eigenvalues $\{ \lambda_\epsilon \} $ (omitting indexing) converge to those corresponding to $\{\lambda_0\}$  for the "homogenized" transmission eigenvalue problem, that is, those corresponding to 
\begin{alignat}{5}\label{tep0}
\Delta w+\tau \overline{n} w &= 0 &\qquad & \mathrm{in} \ D \\
\Delta v+\tau v& =0 & &  \mathrm{in}\  D \\
w&= v  & &\mathrm{on} \ \partial D \\
\displaystyle \frac{\partial w}{\partial \nu}&=\displaystyle \frac{\partial v}{\partial \nu} & & \mathrm{on}\  \partial D.\label{tep10} 
\end{alignat}
where  $\overline{n}$ denotes the period cell average 
$$ \overline{n} = {1\over{|Y|}} \int_Y n(y) dy .$$
Our motivation here is to find the next order term, i.e. the corrections $\tau^{(1)}$ , where each $$ \tau_\epsilon = \tau_0 +\epsilon \tau^{(1)} + o(\epsilon).$$ 

\noindent
From the work \cite{CaGiHa}, for this setup we have that the transmission eigenvalue problem (\ref{tep}) is equivalent to the fourth order nonlinear eigenvalue problem: Find $\tau$ such that there exist  nontrivial  $u=w-v\in H^2_0(D)$  such that 
\begin{equation}\label{10}
(\Delta +  \tau n_\epsilon) \frac{1}{n_\epsilon -1}(\Delta + \tau) u = 0,
\end{equation}
where we use $n_\epsilon$ denote the periodic $$n_\epsilon = n(x/\epsilon).$$
We can state this in variational form follows: Find $u\in H^2_0(D)$ such that
\begin{equation}\label{var}
\int_D \frac{1}{n_\epsilon -1}(\Delta u + \tau u)(\Delta \phi +\tau  n_\epsilon \phi )\, \mathrm{d}x = 0 \quad \text{for all $\phi\in H^2_0(D)$.}
\end{equation}
Following \cite{CaGiHa}, we rewrite this in terms of variationally defined operators. Let  us  define the bounded bilinear forms on $H^2_0(D)\times H^2_0(D)$,
\begin{equation} \label{AA}
\mathcal{A}_{\tau,\epsilon}(u,\phi)= \left ( 
\frac{1}{n_\epsilon-1}(\Delta u + \tau u), (\Delta \phi + \tau \phi)
\right )_{L^2(D)} + \tau^2 (u,\phi)_{L^2(D)}
\end{equation} and 
\begin{equation}
\mathcal{B}(u,\phi) = (\nabla u, \nabla \phi)_{L^2(D)}.
\end{equation}
By the Riesz Representation Theorem, these bilinear forms define bounded operators $\mathbb{A}_{\tau,\epsilon}: H^2_0(D)\rightarrow H^2_0(D)$ and $\mathbb{B}: H^2_0(D)\rightarrow H^2_0(D) $ which are such that 
\begin{equation}\label{Adef}
\mathcal{A}_{\tau,\epsilon}(u,\phi)=(\mathbb{A}_{\tau,\epsilon}u,\phi)_{H^2_0(D)}\quad
 \text{ and }\quad 
\mathcal{B}(u,\phi)=(\mathbb{B} u,\phi)_{H^2_0(D)}
\end{equation}
for all $u,\phi\in H^2_0(D)$. We may also find it convenient to write these variationally defined operators using pde notation. For given $f\in H^2(D)$ we have that 
$$ \mathbb{A}_{\tau,\epsilon}^{-1} f  = \left( (\Delta +\tau ) \frac{1}{n_\epsilon-1} (\Delta +\tau ) +\tau^2 \right)^{-1}\Delta\Delta f $$
and 
$$ \mathbb{B}f = (\Delta\Delta)^{-1}\Delta f $$
where the inverses of the fourth order operators have range in $H^2_0(D)$, where solutions are unique. We also note that the operator $\mathbb{B}: H^2( D)\rightarrow H^2_0( D)$ has a bounded extension on $L^2( D)$; for any $f\in L^2( D)$, $\Delta f$ is understood in the sense of $H^{-2}(D)$, the dual of $H^2_0(D)$. We continue to use $\mathbb{B}$ to denote this operator $\mathbb{B}:L^2( D)\rightarrow H^2_0( D)$, so that $\mathbb{B}$ is clearly compact from $L^2( D)$ to itself. Furthermore, $\mathbb{A}_{\tau,\epsilon}$ is invertible on $H^2_0(D)$ for positive real $\tau$, and the coercivity constant is independent of $\tau$  \cite{CaHa}. 
The variational form (\ref{var}) of the transmission eigenvalue problem is equivalent to finding $u\in L^2(D)$ such that
\begin{equation}\label{evprob}
(I-\tau \mathbb{A}_{\tau,\epsilon}^{-1}\mathbb{B})u=0.
\end{equation}
Define the linear operator $T_\epsilon(\tau): L^2(D)\to L^2(D)$ for $\epsilon\geq 0$ and $\tau\in \mathbb{C}$ such that
\begin{eqnarray}\label{Te}
 T_\epsilon (\tau) &:= & \mathbb{A}_{\tau,\epsilon}^{-1} \mathbb{B}, \\ 
& = &\left( (\Delta +\tau ) \frac{1}{n_\epsilon-1} (\Delta +\tau ) +\tau^2 \right)^{-1}\Delta  \nonumber
 \end{eqnarray}
so we can write (\ref{evprob}) as
 \begin{equation} \label{TEepsilon}
 \tau T_\epsilon (\tau) u = u,
 \end{equation}
 and its limiting problem
 \begin{equation} \label{TE0}
 \tau T_0 (\tau) u = u.
 \end{equation}
 Here $T_0 $ is defined as 
 \begin{eqnarray}\label{T0}
 T_0 (\tau) &:=&  \mathbb{A}_{\tau,0}^{-1} \mathbb{B}, \\
 & = &\left( (\Delta +\tau ) \frac{1}{\overline{n}-1} (\Delta +\tau ) +\tau^2 \right)^{-1}\Delta  \nonumber
 \end{eqnarray}
 where $A_{\tau,0}$ is defined as in (\ref{Adef}) but with $n_\epsilon$ replaced with its limiting value $n_0=\overline{n}$. 

\noindent
We have now rephrased the problem as a nonlinear eigenvalue perturbation problem. That is, a transmission eigenvalue $\tau_\epsilon$  is a value for $\tau$ such that there exists a nontrivial $u\in L^2(D)$ satisfying
$$\tau_\epsilon T_\epsilon(\tau_\epsilon) u = u$$
for $\epsilon>0$. A limiting transmission eigenvalue is a value $\tau_0$ such that there exists nontrivial $u\in L^2(D)$ such that
$$\tau_0 T_0(\tau_0) u = u.$$
In order to find a correction formula for the transmission eigenvalues of the perturbed problem in terms of the eigenvalues and eigenvectors of the background problem, we will apply a result in \cite{Mo}, an application of Osborn's theorem for approximating the eigenvalues of compact operators \cite{Os}.  
Let $\tau_\epsilon$ be the eigenvalue associated with the nonlinear eigenvalue problem \eqref{TEepsilon}
and let $\tau_0$ be the eigenvalue corresponding to the limiting eigenvalue problem
\eqref{TE0}.
In this paper we will derive an expression for the next order correction term $\tau^{(1)}$ in the asymptotic expansion
\begin{equation*}
\tau_\epsilon = \tau_0+ \epsilon\tau^{(1)} + o(\epsilon ).
\end{equation*} 
\begin{remark}
{\em For sake of presentation, we present the calculations for the case when $n_\epsilon>1$ uniformly in $D$. Similar analysis can be done in the case $0<n_\epsilon<1$ uniformly in $D$. In this case the definition of the coercive part  (\ref{AA}) is replaced  (see \cite[Section 4.2]{CCH-book})
$$
\tilde{\mathcal{A}}_{\tau,\epsilon}(u,\phi)= \left ( 
\frac{1}{1-n_\epsilon}(\Delta u + \tau n_\epsilon u), (\Delta \phi + \tau n_\epsilon  \phi)
\right )_{L^2(D)} + \tau^2 (n_\epsilon u,\phi)_{L^2(D)}
$$
with the corresponding operators defined accordingly.}
\end{remark}
\section{Operator convergence: a fourth order homogenization problem}\label{sec3}
In order to apply the eigenvalue correction theorem, we will need to explore the convergence of $T_\epsilon(\tau)$ to $T_0(\tau)$, or more precisely, we will need an asymptotic expansion  with respect to $\epsilon$ for $T_\epsilon(\tau)$ and corresponding norm estimates. We need to focus on $ \mathbb{A}_{\tau,\epsilon}^{-1}$, since all of the $\epsilon$ dependence is in this operator. Note that if $$u_\epsilon = \mathbb{A}_{\tau,\epsilon}^{-1} f, $$
then  $u_\epsilon\in H^2_0(D)$ is  the variational solution to  
\begin{equation} \label{eq:uepsn} (\Delta +\tau ) {1\over{n(x/\epsilon) -1}}(\Delta +\tau )u_\epsilon +\tau^2 u_\epsilon = \Delta\Delta f \mbox{ in } D \end{equation}
where $n(y)$ is periodic on the period cell $Y$.  For simplicity of exposition, we let 
\begin{equation} \label{eq:a} a(x/\epsilon) = {1\over{n(x/\epsilon) -1}}  \end{equation}
and $$h= \Delta \Delta f ,$$
so that $u_\epsilon\in H^2_0(D)$ solves 
\begin{equation} \label{eq:uepsa} (\Delta +\tau ) a(x/\epsilon) (\Delta +\tau )u_\epsilon +\tau^2 u_\epsilon = h\mbox{ in } D ;\end{equation}
 the periodic homogenization problem which is the subject of this section.  We note that such fourth order periodic problems have been studied in the past, see for example, \cite{Pa},\cite{Fr},   and so the expansion of the main part of the operator is not new. Here we focus on obtaining high enough order $L^2$ norm estimates, which we will need to apply the eigenvalue perturbation theorem. These estimates require the introduction and analysis of a fourth order boundary corrector function. 

\subsection{Formal asymptotics} We proceed by assuming that $a(y)\geq a_0>0$  is positive, bounded in $L^\infty$ and periodic, and we do standard two-scale asymptotic expansions. Let $y=x/\epsilon$ so that from the chain rule 
$$\nabla = \nabla_x +{1\over{\epsilon}} \nabla_y , $$
and assume the ansatz for the solution $u_\epsilon$ of (\ref{eq:uepsa})
\begin{equation} \label{eq:ansatz} u_\epsilon \approx u_0(x,y) + \epsilon u^{(1)} (x,y) +\epsilon^2 u^{(2)}(x,y) +\ldots \end{equation}
where each $u^{(i)}$ is periodic in the fast variable $y$ in the sense of $H^2_\sharp(Y)$, where $H^2_\sharp(Y)$ is defined to be $H^2$ functions on the torus, defined in terms of the decay of the Fourier coefficients. Equivalently, this is all $H^2(Y)$ functions which are also in $H^2$ across the matching boundaries (the closure of smooth functions on the torus in the $H^2$ norm). We note that 
$$ \Delta = \Delta_x + {2\over{\epsilon}} \nabla_y\cdot\nabla_x + {1\over{\epsilon^2}} \Delta_y. $$
We could proceed by plugging the ansatz (\ref{eq:ansatz}) into (\ref{eq:uepsa}), however, we will instead rewrite (\ref{eq:uepsa}) as a second order system. The use of a lower order system both simplifies the derivation of the terms in the ansantz and potentially allows for lower regularity assumptions.  To this end, we let $ v_\epsilon = a(x/\epsilon) (\Delta +\tau )u_\epsilon $ so that the pair $(u_\epsilon, v_\epsilon)$ solves
\begin{eqnarray}\label{eq:uepssystem}  a(x/\epsilon) (\Delta +\tau )u_\epsilon  &=& v_\epsilon  \\ (\Delta +\tau )v_\epsilon  +\tau^2 u_\epsilon &=& h, \nonumber\end{eqnarray}
and so we also expand 
\begin{equation} \label{eq:ansatzv} v_\epsilon \approx v_0(x,y) + \epsilon v^{(1)} (x,y) +\epsilon^2 v^{(2)}(x,y) +\ldots .\end{equation}
We plug the ansatz into the system, 
\begin{multline}  a(y) (\Delta_x + {2\over{\epsilon}} \nabla_y\cdot\nabla_x + {1\over{\epsilon^2}} \Delta_y +\tau ) ( u_0 +\epsilon u^{(1)} +\epsilon^2 u^{(2)} +\ldots ) = v_0 +\epsilon v^{(1)} +\epsilon^2 v^{(2)} +\ldots \\ (\Delta_x + {2\over{\epsilon}} \nabla_y\cdot\nabla_x + {1\over{\epsilon^2}} \Delta_y +\tau ) ( v_0 +\epsilon v^{(1)} +\epsilon^2 v^{(2)} +\ldots ) \\ + \tau^2( u_0 +\epsilon u^{(1)} +\epsilon^2 u^{(2)} +\ldots ) = h ,\end{multline} 
and set equal the coefficients of like powers of epsilon to obtain the equations
\begin{eqnarray}\label{eq:eps-2} {1\over{\epsilon^2}}: \quad  a\Delta_y u_0 &=& 0 \\ \Delta_y v_0 &=& 0 \nonumber\end{eqnarray}
\begin{eqnarray}\label{eq:eps-1} {1\over{\epsilon}}: \quad  2a \nabla_x \nabla_y u_0 +  a \Delta_y u^{(1)} &=& 0 \\ 2\nabla_x \nabla_y v_0 + \Delta_y v^{(1)} &=& 0 \nonumber \end{eqnarray}
\begin{eqnarray}\label{eq:eps0} {\epsilon^0}:  \quad  a(\Delta_x+\tau)u_0 + 2a \nabla_x\cdot \nabla_y u^{(1)}+  a \Delta_y u^{(2)} &=& v_0 \\  (\Delta_x+\tau)v_0 + 2\nabla_x\cdot \nabla_y v^{(1)} + \Delta_y v^{(2)} +\tau^2 u_0 &=& h\nonumber \end{eqnarray}
\begin{eqnarray}\label{eq:eps1} {\epsilon}:  \quad  a(\Delta_x+\tau)u^{(1)} + 2a \nabla_x\cdot \nabla_y u^{(2)}+  a \Delta_y u^{(3)} &=& v^{(1)} \\  (\Delta_x+\tau)v^{(1)} + 2\nabla_x\cdot \nabla_y v^{(2)} + \Delta_y v^{(3)} +\tau^2 u^{(1)} &=& 0\nonumber \end{eqnarray}
\begin{eqnarray}\label{eq:eps2} {\epsilon^2}:  \quad  a(\Delta_x+\tau)u^{(2)} + 2a \nabla_x\cdot \nabla_y u^{(3)}+  a \Delta_y u^{(4)} &=& v^{(2)} \\  (\Delta_x+\tau)v^{(2)} + 2\nabla_x\cdot \nabla_y v^{(3)} + \Delta_y v^{(4)} +\tau^2 u^{(2)} &=& 0,\nonumber\end{eqnarray}
and in general, for $n\geq 1$, the equation corresponding to $\epsilon^n$ is 
\begin{eqnarray}\label{eq:epsn} {\epsilon^n}:  \quad  a(\Delta_x+\tau)u^{(n)} + 2a \nabla_x\cdot \nabla_y u^{(n+1)}+  a \Delta_y u^{(n+2)} &=& v^{(n)} \\  (\Delta_x+\tau)v^{(n)} + 2\nabla_x\cdot \nabla_y v^{(n+1)} + \Delta_y v^{(n+2)} +\tau^2 u^{(n)} &=& 0.\nonumber\end{eqnarray}
First we observe that the first two sets of equations (\ref{eq:eps-2}) and (\ref{eq:eps-1}) imply that the first terms do not depend on $y$, that is, $u_0 =u_0(x)$, $v_0=v_0(x)$, $u^{(1)} =u^{(1)}(x)$, and $v^{(1)}(x)=v_0(x)$. Since $v_0(x)$ does not depend on $y$, the first equation in  (\ref{eq:eps0}) suggests that we should take 
\begin{equation} \label{eq:u2} u^{(2)} = \chi(y) (\Delta_x +\tau ) u_0 \end{equation}
where $$a+ a\Delta_y \chi = c$$ for some constant $c$. Periodicity implies that we must have $c =  \overline{a^{-1}}^{-1}$, leading to  
\begin{equation}\label{eq:chi} \Delta_y \chi(y) = \overline{a^{-1}}^{-1}/a -1. \end{equation}
We note that apriori $u^{(2)}$ could still have an additive function of $x$.
Taking the $Y$ cell average of (\ref{eq:eps0}) and using the formula for $\chi$ we find the homogenized problem
  \begin{equation} \label{homog2}  (\Delta_x +\tau) \overline{a^{-1}}^{-1} (\Delta_x +\tau) u_0 +\tau^2 u_0= h , \end{equation}
  accompanied by 
  \begin{equation} \label{eq:v0} v_0 = \overline{a^{-1}}^{-1}(\Delta_x +\tau) u_0.\end{equation}
Now, if we take 
\begin{equation} \label{eq:u3} u^{(3)}= \vec{\gamma}(y)\cdot \nabla_x (\Delta_x +\tau) u_0  \end{equation}
where the vector $\vec{\gamma}$ has cell average zero and solves 
\begin{equation}\label{gammaeq} -\Delta_y \vec{\gamma}  =  2\nabla_y\chi(y) , \end{equation}
we see that the first equation of (\ref{eq:eps1}) is satisfied with $$u^{(1)}=v^{(1)}=0.$$ The second equation of (\ref{eq:eps1}) is also satisfied if we take $$v^{(2)}=v^{(3)}=0.$$ If we do this, to satisfy the first equation of (\ref{eq:eps2}) we can take 
\begin{equation} \label{eq:u4} u^{(4)}= \alpha(y) (\Delta_x +\tau)  (\Delta_x +\tau) u_0  + B_{ij}(y) (D^2_x)_{ij}  (\Delta_x +\tau) u_0\end{equation}
where the $\alpha(y)$ has cell average zero and solves 
\begin{equation}\label{alphaeq} -\Delta_y  \alpha  =  \chi , \end{equation}
and matrix $B(y)$ has components with cell average zero satisfying
\begin{equation} \label{Beq} -\Delta_y B_{ij} = 2{\partial\gamma_i\over{\partial y_j}}.\end{equation}
In (\ref{eq:u4}) Einstein summation notation is employed, with $D^2$ denoting the Hessian.  We find then to satisfy the second equation in (\ref{eq:eps2}) we need a nonzero $v^{(4)}$, and taking
\begin{equation}\label{eq:v4} v^{(4)} = \tau^2 \alpha(y) (\Delta_x +\tau) u_0  \end{equation}
will work. To summarize, we have thus far derived 
$$ u_\epsilon \approx u_0 +\epsilon^2 u^{(2)} +\epsilon^3 u^{(3)} +\epsilon^4 u^{(4)} \cdots $$
$$ v_\epsilon \approx v_0 + \epsilon^4 v^{(4)}+\cdots, $$
with $u^{(2)}$, $u^{(3)}$,$u^{(4)}$, $v_0$,  and $v^{(4)}$ given by (\ref{eq:u2}), (\ref{eq:u3}), (\ref{eq:u4}), (\ref{eq:v0}), and (\ref{eq:v4}) respectively. We need to emphasize, however, that beyond second order these choices are not necessarily optimal; there may be other third and fourth order terms necessary if one wanted estimates of higher order. 

\noindent
Our solutions $u_\epsilon$ and $u_0$ are in $H^2_0(D)$,  but due to the corrections, our approximation to $u_\epsilon$ is no longer in $H^2_0(D)$. Hence in order to obtain high enough order convergence estimates, will need the boundary corrector functions at each order. Let $\theta_\epsilon^{(n)}$ denote the unique $H^2(D)$ solution to 
\begin{eqnarray}   (\Delta_x +\tau) a(x/\epsilon) (\Delta +\tau) \theta_\epsilon^{(n)} +\tau^2\theta_\epsilon^{(n)} &=&  0 \ \ \ \mbox{in}\ \ D \label{thetaeps2} \\ 
\theta_\epsilon^{(n)} &=& -\epsilon u^{(n)} \ \ \mbox{on}\ \ \partial D \nonumber \\ {\partial  \theta_\epsilon^{(n)} \over{\partial\nu}}&=& -\epsilon {\partial u^{(n)}\over{\partial\nu}}  \ \ \mbox{on}\ \ \partial D,\nonumber\end{eqnarray}
and define its second order system counterpart 
\begin{equation}\label{eq:psieps} \psi_\epsilon^{(n)} = a(x/\epsilon) (\Delta +\tau) \theta_\epsilon^{(n)} .\end{equation}
Then the pair $(\theta^{(n)}_\epsilon, \psi^{(n)}_\epsilon)$ solves
\begin{eqnarray}\label{eq:uepssystem}  a(x/\epsilon) (\Delta +\tau )\theta^{(n)}_\epsilon  &=& \psi^{(n)}_\epsilon  \\ (\Delta +\tau )\psi^{(n)}_\epsilon  +\tau^2 \theta^{(n)}_\epsilon &=& 0. \nonumber\end{eqnarray}
\begin{remark} \label{rem} Since for our transmission eigenvalue problem $a= 1/(n-1)$, we have that 
\begin{equation}\label{betaeq2}  \Delta_y \chi(y) = {n(y) -\overline{n}\over{\overline{n}-1}} , \end{equation}
which means that $$\chi ={\beta \over{\overline{n}-1}},$$ where $\beta$ is the first order cell function from the homogenization of the standard transmission problem corresponding to $n$; which is $Y$-periodic, has cell average zero, and solves \begin{equation}\label{betaeq3}  \Delta_y \beta(y) = {n(y) -\overline{n}} , \end{equation}
see for example \cite{CaGuMoPa},\cite{CaMoPa}.
\end{remark} 
\subsection{Norm estimates}
The following Lemma will be useful for showing convergence estimates.
\begin{lemma} \label{systemlem} Assume that $z_\epsilon, \eta_\epsilon$ are in $H^2_0( D)$ and $L^2( D)$ respectively, and that they satisfy the second order system 
\begin{eqnarray}\label{eq:zepssystem}  a(x/\epsilon) (\Delta +\tau )z_\epsilon - \eta_\epsilon  &= & e \\ (\Delta +\tau )\eta_\epsilon  +\tau^2 z_\epsilon &=& f. \nonumber\end{eqnarray}
Then there exists $C$ independent of $\epsilon$ such that 
\begin{equation} \| z_\epsilon\|_{H^2_0( D)} \leq C (\| e \|_{L^2( D)} + \| f\|_{H^{-2}( D)}). \end{equation} 
\end{lemma}
\begin{proof}
Consider \begin{eqnarray} \int_ D a(\Delta + \tau)z_\epsilon (\Delta +\tau)z_\epsilon &=& \int_ D \eta_\epsilon (\Delta +\tau)z_\epsilon + \int_ D e(\Delta+\tau)z_\epsilon \nonumber \\ &=& \int_ D  (\Delta +\tau)\eta_\epsilon z_\epsilon + \int_ D e(\Delta+\tau)z_\epsilon \nonumber \\ 
&=&-\tau^2 \int_ D z_\epsilon z_\epsilon +\int_ D f z_\epsilon + \int_ D e(\Delta+\tau)z_\epsilon \nonumber \end{eqnarray}
where in the second line we integrated by parts and used the fact that $z_\epsilon$ has zero boundary data. Using ellipticity and Cauchy-Schwartz we have
 \begin{eqnarray} c\| z_\epsilon\|^2_{H^2_0( D)} &\leq& \int_ D a(\Delta + \tau)z_\epsilon (\Delta +\tau)z_\epsilon +\tau^2 \int_ D z_\epsilon z_\epsilon \nonumber \\ &= &\int_ D f z_\epsilon + \int_ D e(\Delta+\tau)z_\epsilon \nonumber \\ &\leq& \| f\|_{H^{-2}( D)} \| z_\epsilon\|_{H^2_0( D)} + \| e \|_{L^2( D)}   \| z_\epsilon\|_{H^2_0( D)} .\end{eqnarray}
Dividing through by $\| z_\epsilon\|_{H^2_0( D)}$ the result follows.
\end{proof}
\noindent
The next result gives us first order convergence in $\epsilon$, which we will need to show convergence of the operators. 
\begin{proposition}\label{h2est} Let $u_\epsilon, u_0 \in H^2_0( D)$  be the solutions to 
\begin{equation} \label{fourthordeqest}  (\Delta_x +\tau) a(x/\epsilon) (\Delta +\tau) u_\epsilon +\tau^2 u_\epsilon= h \end{equation}
and
\begin{equation} \label{homogeqest}  (\Delta_x +\tau) \overline{a^{-1}}^{-1} (\Delta +\tau) u_0 +\tau^2 u_0= h \end{equation}
respectively. Then 
$$ \| u_\epsilon -(u_0 +\epsilon^2 u^{(2)}  +\epsilon \theta_\epsilon^{(2)}) \|_{H_0^2( D)} \leq C\epsilon \| u_0 \|_{H^4( D)} $$
where $u^{(2)}$ is given by (\ref{eq:u2}) and where the boundary correction $\theta^{(2)}_\epsilon$ is defined by (\ref{thetaeps2}) for $n=2$. 
\end{proposition} 
\begin{proof} Let 
$$ z_\epsilon =  u_\epsilon -(u_0 +\epsilon^2 u^{(2)}  +\epsilon\theta_\epsilon^{(2)} ) $$
and $$ \eta_\epsilon = v_\epsilon -(v_0 + \epsilon  \psi_\epsilon^{(2)}  )$$
where $\psi_\epsilon^{(2)}$ is given by (\ref{eq:psieps}) and we recall that $v_0$ is given by (\ref{eq:v0}).  Thanks to the boundary corrections, $z_\epsilon\in H^2_0( D)$. We calculate 
\begin{equation} a(x/\epsilon) (\Delta +\tau )z_\epsilon - \eta_\epsilon  = -2\epsilon  a \nabla_x\cdot\nabla_y u^{(2)} - \epsilon^2 a(\Delta_x+\tau)u^{(2)}  \label{useforeta1} \end{equation}
and
\begin{equation} (\Delta +\tau )\eta_\epsilon  +\tau^2 z_\epsilon = -\epsilon^2 \tau^2 u^{(2)} .\label{useforeta2} \end{equation}
The residual contains derivatives of $u_0$ of fourth order or lower, and so the result follows from Lemma \ref{systemlem} . 
\end{proof}
\noindent
\begin{corollary} \label{h2cor} Let $u_\epsilon, u_0 \in H^2_0( D)$  be the solutions to 
\begin{equation} \label{fourthordeqest2}  (\Delta_x +\tau) a(x/\epsilon) (\Delta +\tau) u_\epsilon +\tau^2 u_\epsilon= h \end{equation}
and
\begin{equation} \label{homogeqest2}  (\Delta_x +\tau) \overline{a^{-1}}^{-1} (\Delta +\tau) u_0 +\tau^2 u_0= h \end{equation}
respectively. Then the boundary correction $\theta_\epsilon^{(2)}$  given by (\ref{thetaeps2}) with $n=2$ satisfies 
$$\| \theta_\epsilon^{(2)}\|_{H^2( D)}\leq C\epsilon^{-1/2}\| u_0\|_{H^4( D)}$$
and hence 
$$ \| u_\epsilon - u_0 -\epsilon^2 u^{(2)}\|_{H^2( D)} \leq C\epsilon^{1/2} \| u_0 \|_{H^4( D)} .$$
where $u^{(2)}$ is given by (\ref{eq:u2}). 
\end{corollary} 
\begin{proof} From direct calculation of derivatives on the boundary and standard interpolation we get that $$ \| \epsilon u^{(2)} \|_{H^{3/2}(\partial D)}\leq C\epsilon^{-1/2} \| u_0 \|_{H^4( D)}$$ and $$ \| \epsilon{ \partial u^{(2)}\over{\partial\nu}} \|_{H^{1/2}(\partial D)}\leq C\epsilon^{-1/2} \| u_0 \|_{H^4( D)},$$ and so the bound on the boundary corrector follows from standard elliptic estimates. The result then follows from Proposition \ref{h2est}.
\end{proof}
\noindent
To get higher order estimates we need to use further terms in the asymptotic expansion. 
\begin{proposition}\label{h2est2} Let $u_\epsilon, u_0 \in H^2_0( D)$  be the solutions to 
\begin{equation} \label{fourthordeqest}  (\Delta_x +\tau) a(x/\epsilon) (\Delta +\tau) u_\epsilon +\tau^2 u_\epsilon= h \end{equation}
and
\begin{equation} \label{homogeqest}  (\Delta_x +\tau) \overline{a^{-1}}^{-1} (\Delta +\tau) u_0 +\tau^2 u_0= h \end{equation}
respectively. Then 
$$ \| u_\epsilon -(u_0 +\epsilon^2 u^{(2)}  +\epsilon\theta_\epsilon^{(2)}+ \epsilon^3 u^{(3)}  +\epsilon^2\theta_\epsilon^{(3)} ) \|_{H_0^2( D)} \leq C\epsilon^2 \| u_0 \|_{H^5( D)} $$
where $u^{(2)}$ and $u^{(3)}$are given by (\ref{eq:u2}) and (\ref{eq:u3}), and where the boundary corrections $\theta^{(n)}_\epsilon$ are defined by (\ref{thetaeps2}) for $n=2,3$. 
\end{proposition} 
\begin{proof} Let 
$$ z_\epsilon =  u_\epsilon -(u_0 +\epsilon^2 u^{(2)}  +\epsilon \theta_\epsilon^{(2)} +\epsilon^3  u^{(3)}  +\epsilon^2\theta_\epsilon^{(3)}  ) $$
and $$ \eta_\epsilon = v_\epsilon -(v_0 + \epsilon  \psi_\epsilon^{(2)}+  \epsilon^2  \psi_\epsilon^{(3)} )$$
where $\psi_\epsilon^{(2)}$ is given by (\ref{eq:psieps}) and we recall that $v_0$ is given by (\ref{eq:v0}).  Thanks to the boundary corrections, $z_\epsilon\in H^2_0( D)$. We calculate 
\begin{equation} a(x/\epsilon) (\Delta +\tau )z_\epsilon - \eta_\epsilon  = -2\epsilon^2  a \nabla_x\cdot\nabla_y u^{(3)} - \epsilon^2 a(\Delta_x+\tau)u^{(2)}   -\epsilon^3 a(\Delta_x+\tau)u^{(3)}\end{equation}
and
\begin{equation} (\Delta +\tau )\eta_\epsilon  +\tau^2 z_\epsilon = -\epsilon^2 \tau^2 u^{(2)} .\end{equation}
Again the residual contains up to fifth derivatives of $u_0$, and so the result follows again from Lemma \ref{systemlem}. 
\end{proof}
\noindent
We have that the boundary corrector terms, $\epsilon^{n-1}\theta_\epsilon^{(n)}$, with $\theta_\epsilon^{(n)}$ given by (\ref{thetaeps2}), are of order $\epsilon^{n-1}$ in general. 
\begin{lemma}\label{bcbounds} The boundary correctors $\theta_\epsilon^{(n)}$ given by (\ref{thetaeps2})  for $n=2,3$ satisfy 
$$\| \theta_\epsilon^{(n)}\|_{L^2( D)} \leq C\| u_0 \|_{H^{n+2}( D)} $$
where $C$ is independent of $\epsilon$ and $u_0$. \end{lemma}
\begin{proof}  We prove this for $n=2$, the proof for $n=3$ follows in the same way. 
Given any $f\in L^2( D)$, consider the solution $w_\epsilon\in H^2_0( D)$ of 
\begin{eqnarray}  \label{wepseq} (\Delta_x +\tau) a(x/\epsilon) (\Delta +\tau) w_\epsilon +\tau^2 w_\epsilon&=& f \ \ \ \mbox{in}\ \  D \\ 
w_\epsilon &=& 0 \ \ \mbox{on}\ \ \partial D \nonumber \\ {\partial w_\epsilon} \over{\partial\nu}&=& 0 \ \ \mbox{on}\ \ \partial D.\end{eqnarray}
Using the equations for $\theta_\epsilon^{(2)}$ and $w_\epsilon$ and the second Green's identity twice,  we obtain
\begin{eqnarray}
\int_ D \theta_\epsilon^{(2)} f &=&\int_{ D} \left[(\Delta_x +\tau) a(x/\epsilon) (\Delta +\tau) w_\epsilon +\tau^2 w_\epsilon\right] \theta_\epsilon^{(2)}\,dx\nonumber\\
&=& \int_{\partial  D}\theta_\epsilon^{(2)} \, \frac{\partial}{\partial\nu} \left(a(x/\epsilon) (\Delta +\tau) w_\epsilon\right)- a(x/\epsilon) (\Delta +\tau) w_\epsilon \,\frac{\partial \theta_\epsilon^{(2)}}{\partial \nu} \, ds.
\end{eqnarray}
The boundary conditions for $\theta_\epsilon^{(2)}$  then yield
\begin{multline}\label{eqthetaafterbcs}
\int_ D \theta_\epsilon^{(2)} f =  \int_{\partial  D}\epsilon\chi(x/\epsilon)(\Delta +\tau)u_0 \, \frac{\partial}{\partial\nu}\left(  a(x/\epsilon) (\Delta +\tau) w_\epsilon\right)\\ - \int_{\partial D} \epsilon a(x/\epsilon) (\Delta +\tau) w_\epsilon \,\frac{\partial}{\partial\nu} \left(\chi(x/\epsilon)(\Delta +\tau)u_0\right) \, ds.
\end{multline}
From Proposition \ref{h2est} applied to the homogenization problem for $w_\epsilon$, we get have that 
\begin{equation}\label{l2estlapwiththeta} \| w_\epsilon - w_0 -\epsilon^2 w^{(2)} - \epsilon \theta^f_\epsilon \|_{H^2( D)} \leq C\epsilon \| w_0\|_{H^4( D)} \end{equation}
where $w_0$ is the homogenized solution for (\ref{wepseq}), $w^{(2)}$ is the corresponding bulk correction, and $\theta^f_\epsilon$ is its corresponding boundary corrector (for order $n=2$). 
From line (\ref{useforeta1}) in the proof of the same Proposition, we have that $\eta_\epsilon$ is $O(\epsilon)$ in $L^2$. Likewise, $ \Delta \eta_\epsilon $  is also  bounded by the same right hand side in $L^2(D)$ from (\ref{useforeta2}).  Hence we have that 
$$ \|  \eta_\epsilon \|_{L^2(D,\Delta)} \leq  C\epsilon \| w_0\|_{H^4( D)}, $$
where we use the space (see for example \cite{Va})  $$L^2(D, \Delta) = \{ v \in L^2(D)| \Delta v\in L^2(D)\}  $$ with norm $$ \| v\|_{L^2(D,\Delta)} = \|v\|_{L^2(D)} + \| \Delta v \|_{L^2(D)}.$$ It is known, (see Appendix \ref{trace}), that $L^2(D,\Delta)$ has bounded boundary traces in $H^{-1/2}$ and bounded normal derivative boundary traces in $H^{-3/2}$. From this and the proof of Proposition \ref{h2est} we can conclude that 
\begin{equation}\label{hminushalfest} \|a (\Delta+\tau)w_\epsilon -\overline{a^{-1}}^{-1}(\Delta +\tau)w_0 -\epsilon a (\Delta+\tau)\theta^f_\epsilon \|_{H^{-1/2}(\partial D)} \leq C\epsilon \| w_0\|_{H^4( D)} \end{equation} and 
\begin{equation}\label{hminus3halfest} \| {\partial\over{\partial\nu}} (a(\Delta+\tau)w_\epsilon) -\overline{a^{-1}}^{-1}{\partial\over{\partial\nu}}(\Delta +\tau)w_0 -\epsilon {\partial\over{\partial\nu}} (a (\Delta+\tau)\theta^f_\epsilon) \|_{H^{-3/2}(\partial D)} \leq C\epsilon \| w_0\|_{H^4( D)}. \end{equation}
Thanks to these estimates, we can replace the $w_\epsilon$ terms in (\ref{eqthetaafterbcs}), with the remainder bounded by $$C \epsilon^2 \| \chi(x/\epsilon)(\Delta + \tau) u_0 \|_{H^{3/2}(\partial D)}   \| w_0\|_{H^4(D)}$$ for the first term and  $$ C \epsilon^2 \| {\partial\over{\partial\nu}}\chi(x/\epsilon)(\Delta + \tau) u_0 \|_{H^{1/2}(\partial D)}  \| w_0\|_{H^4(D)} $$  for the second term. Both of these are bounded by $ C\epsilon^{1/2} \| u_0\|_{H^4(D)} \| w_0\|_{H^4(D)}$, where we abuse notation and continue to use $C$ for the constant.  Since we have assumed that $D$ is smooth, we use the standard elliptic estimate that $$\| w_0\|_{H^4(D)}\leq C \| f \|_{L^2(D)},$$ so the remainder  is bounded by $C\epsilon^{1/2}\| f \|_{L^2(D)} \| u_0\| _{H^4(D)}$. Hence (\ref{eqthetaafterbcs}) becomes 
\begin{multline}\label{eqthetaafterreplacement}
\int_ D \theta_\epsilon^{(2)} f =  \int_{\partial  D}\epsilon\chi(x/\epsilon)(\Delta +\tau)u_0 \, \frac{\partial}{\partial\nu}\left(  \overline{a^{-1}}^{-1} (\Delta +\tau) w_0\right)ds\\ \int_{\partial  D}\epsilon^2\chi(x/\epsilon)(\Delta +\tau)u_0 \, \frac{\partial}{\partial\nu}\left(  a(x/\epsilon) (\Delta +\tau) \theta^f_\epsilon\right)ds \\ - \int_{\partial D} \epsilon  \overline{a^{-1}}^{-1}  (\Delta +\tau) w_0\,\frac{\partial}{\partial\nu} \left(\chi(x/\epsilon)(\Delta +\tau)u_0\right) \, ds \\- \int_{\partial D} \epsilon^2 a(x/\epsilon) (\Delta +\tau) \theta^f_\epsilon \,\frac{\partial}{\partial\nu} \left(\chi(x/\epsilon)(\Delta +\tau)u_0\right) \, ds +o(1) \\ = I+II +III +IV +o(1)
\end{multline}
where the tail is bounded in absolute value by $C\epsilon^{1/2}\| f \|_{L^2(D)} \| u_0\| _{H^4(D)}$. The first term $I$ is clearly bounded by the same, as it in fact goes to zero $O(\epsilon)$. For the third term $III$, the normal derivative produces a $1/\epsilon$ when applied to $\chi$, which cancels with the $\epsilon$, yielding that the absolute value of $III$ is bounded by $C\| f \|_{L^2(D)} \| u_0\| _{H^4(D)}$.  For the other two terms, we note that $ a(x/\epsilon) (\Delta +\tau) \theta^f_\epsilon $ is bounded by $ C\epsilon^{-1/2} \| w_0 \|_{H^4(D)}$ in $L^2(D,\Delta)$ by Corollary \ref{h2cor} and the equation for $\theta^f_\epsilon$.  Hence we have by trace estimates for $L^2(D,\Delta)$  (see for example Appendix \ref{trace})  that $$ \| a(x/\epsilon) (\Delta +\tau) \theta^f_\epsilon \|_{H^{-1/2}(\partial D)} \leq C\epsilon^{-1/2} \| f \|_{L^2(D)} $$ and
$$ \| {\partial\over{\partial\nu}} a(x/\epsilon) (\Delta +\tau) \theta^f_\epsilon \|_{H^{-3/2}(\partial D)} \leq C\epsilon^{-1/2} \| f \|_{L^2(D)}. $$
Using the duality pairing, 
\begin{multline} | II | \leq  \epsilon^2 \| \chi(\Delta +\tau)u_0 \|_{H^{3/2}(\partial D)}  \| {\partial\over{\partial\nu}} a(x/\epsilon) (\Delta +\tau) \theta^f_\epsilon \|_{H^{-3/2}(\partial D)} \\ \leq  C\epsilon^2 \epsilon^{-3/2} \| u_0\|_{H^4(D)} \epsilon^{-1/2} \| f \|_{L^2(D)} \end{multline} 
and \begin{multline}  | IV | \leq  \epsilon^2 \| {\partial\over{\partial\nu}} \chi(\Delta +\tau)u_0 \|_{H^{1/2}(\partial D)}  \| a(x/\epsilon) (\Delta +\tau) \theta^f_\epsilon \|_{H^{-1/2}(\partial D)} \\ \leq  C\epsilon^2 \epsilon^{-3/2} \| u_0\|_{H^4(D)} \epsilon^{-1/2} \| f \|_{L^2(D)} \end{multline} 
from which we can conclude that 
$$ \left| \int_ D \theta_\epsilon^{(2)} f \right| \leq C\| u_0\|_{H^4(D)}  \| f \|_{L^2(D)}$$
from which the result follows. 
 \end{proof}
\begin{corollary} \label{l2estexpansion} Let $u_\epsilon, u_0 \in H^2_0( D)$  be the solutions to 
\begin{equation} \label{fourthordeqest2}  (\Delta_x +\tau) a(x/\epsilon) (\Delta +\tau) u_\epsilon +\tau^2 u_\epsilon= h \end{equation}
and
\begin{equation} \label{homogeqest2}  (\Delta_x +\tau) \overline{a^{-1}}^{-1} (\Delta +\tau) u_0 +\tau^2 u_0= h \end{equation}
respectively. Then 
$$ \| u_\epsilon -(u_0 + \epsilon\theta_\epsilon^{(2)} ) \|_{L^2( D)} \leq C\epsilon^2 \| u_0 \|_{H^5( D)} $$
and 
$$ \| u_\epsilon - u_0 \|_{L^2( D)} \leq C\epsilon \| u_0 \|_{H^4( D)} .$$
where the boundary correction $\theta_\epsilon^{(2)}$is given by (\ref{thetaeps2}) with $n=2$. 
\end{corollary} 
\begin{proof} The first estimate follows from Proposition \ref{h2est2} and Lemma \ref{bcbounds} applied to $\theta_\epsilon^{(3)}$, and the second follows from Proposition \ref{h2est} and Lemma \ref{bcbounds} applied to $\theta_\epsilon^{(2)}$.
\end{proof}
\section{Transmission eigenvalue expansions} 
The following result about nonlinear eigenvalue perturbations is an extension of a special case of the results in \cite{Os}. This is a slight modification of Corollary~4.1 in \cite{Mo} for the case where we assume only that the operators themselves converge point-wise (strongly), without assuming convergence of the adjoints. The necessary modifications were shown in \cite{FuMo}.  \begin{theorem}
\label{Osbornnonlinear}
Let $X$ be a Hilbert space with sesquilinear inner product $\langle, \rangle$ and $ \{T_\epsilon(\tau):X \rightarrow X\}$ be a set of compact linear operator valued functions of $\tau$ which are analytic in a region $U$ of the complex plane, such that $T_\epsilon(\tau)\rightarrow T_0(\tau)$ pointwise as $\epsilon \rightarrow 0$ uniformly for $\tau \in U$, and that $\{ T_\epsilon(\tau)\}$ are collectively compact, uniformly in $U$.  Let $\tau_0 \neq0$, 
$\tau_0 \in U$ be a simple nonlinear eigenvalue of $T_0$, define ${\mathcal D}T_0(\tau_0)$ to be the derivative of $T_0$ with respect to $\tau$ evaluated at $\tau_0$, and let $\phi$ be a normalized eigenfunction. Then for any $\epsilon$ small enough, there exists $\tau_\epsilon$ a simple nonlinear eigenvalue of $T_\epsilon$, such that if
\begin{equation*}
1+\tau_{0}^2\langle {\mathcal D}T_0(\tau_0) \phi,\phi \rangle \neq 0
\end{equation*}
there exists a constant $C$ independent of $\epsilon$ such that   
\begin{equation}\label{eq:lmh}
\left| \tau_\epsilon - \left(\tau_0+ \dfrac{\tau_{0}^2\langle \left(T_0(\tau_0)-T_\epsilon(\tau_0)\right) \phi,\phi \rangle}{1+ \tau_0^2\langle {\mathcal D}T_0(\tau_0) \phi,\phi \rangle} \right)\right| \leq 
 C \sup_{\tau \in U}\|  (T_0(\tau)-T_\epsilon(\tau))|_{R(E)}\|^2
\end{equation}
where $R(E)$ is the one dimensional eigenspace spanned by $\phi$. 
\end{theorem}
Now let us consider our operators 
\begin{equation}\label{Te2}
 T_\epsilon (\tau): L^2(D) \to L^2(D),\qquad  T_\epsilon (\tau) =  \mathbb{A}_{\tau,\epsilon}^{-1} \mathbb{B}
 \end{equation}
 where  $\mathbb{A}_{\tau,\epsilon}^{-1}: H_0^2(D)\to H^2_0(D)$ is given by 
$$ \mathbb{A}_{\tau,\epsilon}^{-1} f  = \left( (\Delta +\tau ) \frac{1}{n_\epsilon-1} (\Delta +\tau ) +\tau^2 \right)^{-1}\Delta\Delta f $$
and $\mathbb{B}:L^2(D)\to H^2_0(D)$
$$ \mathbb{B}f = (\Delta\Delta)^{-1}\Delta f$$
so that 
\begin{equation}\label{Te3}
 T_\epsilon (\tau) f = \left( (\Delta +\tau ) \frac{1}{n_\epsilon-1} (\Delta +\tau ) +\tau^2 \right)^{-1}\Delta f 
 \end{equation}
 and
 \begin{equation}\label{T03}
 T_0 (\tau) f = \left( (\Delta +\tau ) \frac{1}{\overline{n}-1} (\Delta +\tau ) +\tau^2 \right)^{-1}\Delta f .
 \end{equation}
where the inverses of the fourth order operators have range in $H^2_0(D)$.  We note that $T_\epsilon(\tau)$ and $T_0(\tau)$ are well defined and compact on $L^2(D)$. This follows because $\Delta f $ makes sense in $H^{-2}(D)$, the dual of $H^2_0(D)$, so that the range of both operators is in $H^2_0(D)$, which embeds compactly in $L^2$.  We can therefore take $X=L^2(D)$, with the usual inner product, when applying the above theorem.  

\noindent
For the denominator in the correction theorem, we must compute the derivative of $T_0(\tau)$ with respect to $\tau$, ${\mathcal D}T_0(\tau)$. In fact, this derivative is computed  in  \cite{CaMoRo}, and we include the computations in Appendix \ref{derivative} for the reader's convenience. In our case, the formula (\ref{Lt}) simplifies since our $n_0 =\overline{n}$ is constant.  In particular, the derivative ${\mathcal D}T_0(\tau): L^2(D)\to L^2(D)$ is  given by $\mathcal{D}T_0(\tau)u= -v$ where $v\in H^2_0(D)$ solves 
\begin{equation}\label{dtau}
\Delta\Delta \mathbb{A}_{\tau,0} v=  \frac{2}{\bar n-1}(\Delta + \tau \bar n) \mathbb{A}^{-1}_{\tau,0}\mathbb{B} u.
\end{equation}
Note that the range of $T_0(\tau)$ and its derivative is $H^2_0(D)$. Next we compute $\left<{\mathcal D}T_0(\tau_0) \phi,\phi \right>$ where $\phi$ is an eigenvector corresponding $\tau_0$. To this end, let us denote by $L_{\tau,0}:H^2_0(D)\to H^{-2}(D)$ the mapping
\begin{equation} \label{Ltau0} L_{\tau,0} u =\left( (\Delta +\tau ) \frac{1}{\overline{n}-1} (\Delta +\tau ) +\tau^2 \right)u\end{equation}
Thus, we have $$ T_0 (\tau)=L_{\tau,0}^{-1}\Delta.$$  Note that $L_{\tau,0}$ is coercive, which means that $\left<L_{\tau,0}u,u\right>_{H^{-2}, H^2_0}\geq \alpha \|u\|_{H^2_0}$, and its inverse is well defined with range in $H^2_0(\Omega)$. Recalling also that $\mathbb{A}^{-1}_{\tau_0,0}=L_{\tau,0}^{-1}\Delta\Delta$, we have $\mathbb{A}_{\tau_0,0}=(\Delta\Delta)^{-1} L_{\tau,0}$, and equation (\ref{dtau}) becomes 
\begin{equation}\label{dtau2}
{L}_{\tau,0} v=  \frac{2}{\bar n-1}(\Delta + \tau \bar n) T_0(\tau) u 
\end{equation}
and hence 
\begin{equation}\label{dtau3}
 \mathcal{D}T_0(\tau)u =- {L}_{\tau,0}^{-1} \frac{2}{\bar n-1}(\Delta + \tau \bar n) T_0(\tau) u. 
\end{equation}
If we now take $\tau=\tau_0$ and $u=\phi$ to be the  $L^2$-normalized homogenized transmission eigenfunction corresponding to the  transmission eigenvalue $\tau_0$,  this gives  $$\mathcal{D}T_0(\tau_0)\phi =-\frac{1}{\tau_0}\frac{2}{\bar n-1}L^{-1}_{\tau_0,0}(\Delta + \tau_0 \bar n)\phi,$$
since we know that 
$$  \phi=\tau_0 T_0(\tau_0) \phi. $$ The above calculations yield 
\begin{eqnarray}
1+ \tau_0^2\langle {\mathcal D}T_0(\tau_0) \phi,\phi \rangle &=&1-\frac{2\tau_0}{\bar n-1}\langle L^{-1}_{\tau_0,0}(\Delta + \tau_0 \bar n)\phi,\phi\rangle \label{denomform}\\
&=&1-\frac{2\tau_0}{\bar n-1}\langle T_0(\tau_0) \phi,\phi\rangle -\frac{2\tau^2_0 \bar n}{\bar n-1}\langle L^{-1}_{\tau_0,0}\phi,\phi\rangle\nonumber \\
&=&1-\frac{2}{\bar n-1}-\frac{2\tau^2_0 \bar n}{\bar n-1}\langle L^{-1}_{\tau_0,0}\phi,\phi\rangle\nonumber\\ &=& \frac{1}{\bar n-1}\left(\bar n-3  -2\tau_0\bar n \langle L^{-1}_{\tau_0,0}\phi,\phi\rangle\right).
\nonumber\end{eqnarray}
This expression is obviously non-zero if $1<\bar n\leq 3$, since $L^{-1}_{\tau_0,0}$ is nonnegative and $\tau_0>0$.  To compute  $L^{-1}_{\tau_0,0}\phi$ for the given transmission eigen-pair $(\tau_0,\phi)$ of the homogenized problem, one must solve 
$$(\Delta +\tau_0 ) \frac{1}{\overline{n}-1} (\Delta +\tau_0 )w +\tau_0^2w=\phi \qquad \mbox{for}\; w\in H^2_0(D).$$
Thus it is easy to numerically check if $\left(\bar n-3  -2\tau_0\bar n \langle w,\phi\rangle\right)\neq 0$. In order to  evaluate the numerator we use the asymptotic estimate developed above.  To this end we have 
$$T_0(\tau_0)\phi-T_\epsilon(\tau_0)\phi=L_{\tau_0,0}^{-1}\Delta \phi-L_{\tau_0,\epsilon}^{-1}\Delta \phi$$
where $u_\epsilon:=L_{\tau_0,\epsilon}^{-1}\Delta \phi$  and  $u_0:=L_{\tau_0,0}^{-1}\Delta \phi$ and $u_\epsilon:=L_{\tau_0,\epsilon}^{-1}\Delta \phi$ are the solutions of 
\begin{equation*}  (\Delta +\tau) \frac{1}{n(x/\epsilon)-1}(\Delta +\tau_0) u_\epsilon +\tau_0^2 u_\epsilon= \Delta \phi \end{equation*}
and
\begin{equation*}  (\Delta +\tau) \frac{1}{\bar n-1} (\Delta +\tau_0) u_0 +\tau_0^2 u_0= \Delta \phi,
\end{equation*}
respectively.  From Corollary \ref{l2estexpansion} we have that 
\begin{equation}\label{numform}\left<T_0(\tau_0)\phi-T_\epsilon(\tau_0)\phi, \phi\right>=-\left<\epsilon\theta_\epsilon,\phi\right>+O(\epsilon^2)\end{equation}
where $\theta_\epsilon$ is the solution of 
\begin{eqnarray}   (\Delta +\tau_0) \frac{1}{n(x/\epsilon)-1} (\Delta +\tau_0) \theta_\epsilon^{(2)} +\tau_0^2\theta_\epsilon^{(2)} =  0 && \ \ \ \mbox{in}\ \ D \label{finaltheta} \\ 
\theta_\epsilon^{(2)} = -\epsilon u^{(2)} && \ \ \mbox{on}\ \ \partial D \nonumber \\ {\partial  \theta_\epsilon^{(2)} \over{\partial\nu}}= -\epsilon {\partial u^{(2)}\over{\partial\nu}} &&\ \ \ \mbox{on} \ \ \ \partial D,\nonumber\end{eqnarray}
$$ \mbox{with} \qquad u^{(2)} = \frac{1}{\tau_0}\chi(y) (\Delta_x +\tau_0) \phi   \qquad  \chi ={\beta \over{\overline{n}-1}}$$
where $\beta$ is $Y$-periodic, has cell average zero, and solves  $$\Delta_y \beta(y) = {n(y) -\overline{n}},$$ where we used that  $u_0=\phi/\tau_0$. 
\begin{theorem}
 Assume $n_\epsilon:=n(x/\epsilon)\in L^\infty(D)$ is periodic in $y:=x/\epsilon$ for $y\in Y= [0,1]^d$, and  $n(y)-1$ is positive uniformly in $Y$. Let $\tau_0$ be a simple transmission eigenvalue of the homogenized problem with constant refractive index $\bar n$, and $\phi$ the corresponding eigenfunction normalized such that $\|\phi\|_{L^2(D)}=1$. We assume that $\partial D$ is smooth enough so that $\phi\in H^5(D)$. Then for any $\epsilon>0$ sufficiently small, there exists a simple transmission eigenvalue $\tau_\epsilon>0$ of the periodic media with refractive index $n(x/\epsilon)$, which satisfies the following asymptotic expansion
 \begin{equation}\label{correctedeig}
 \tau_\epsilon=\tau_0+\epsilon\frac{\tau_0^2(1-\bar n)\left<\theta_\epsilon^{(2)}, \phi\right>_{L^2(D)}}{\bar n-3-2\tau_0\bar n \langle L^{-1}_{\tau_0,0}\phi,\phi\rangle_{L^2(D)}}+O(\epsilon^2)
 \end{equation} provided that $\langle L^{-1}_{\tau_0,0}\phi,\phi\rangle_{L^2(D)}\neq \frac{\bar n-3}{2\tau_0\bar n}$,  where $\theta_\epsilon^{(2)}$ is given by (\ref{finaltheta}) and $L_{\tau_0,0}$ is given by (\ref{Ltau0}).
\end{theorem}
\begin{proof}
First we note that Corollary  \ref{l2estexpansion} gives us that for any given $f\in L^2(\Omega)$, \begin{equation} \label{estnow} \| (T_\epsilon(\tau) - T_0(\tau) ) f \|_{L^2(\Omega)} \leq C_\tau \epsilon.\end{equation}  For $\tau$ given in a bounded region of the complex plane, $C_\tau$ can be bounded inependently of $\tau$  from the explicit coercivity of of the fourth order operator \cite{CaHa}. Hence we have strong pointwise convergence of the operators in $L^2(\Omega)$. Furthermore, the operators $\{ T_\epsilon (\tau) \} $ are collectively compact, since $\{ u_\epsilon= T_\epsilon (\tau) f \} $ satisfy $$ \| u_\epsilon \|_{H^2(\Omega)} \leq C \| f\|_{L^2(\Omega)} $$ where $C$ is independent of $\epsilon$ and $\tau$. Hence we can apply Theorem \ref{Osbornnonlinear}. Furthermore, (\ref{estnow}) says that the right hand side of (\ref{eq:lmh}) is $O(\epsilon^2)$, since the eigenspace $R(E)$ is finite dimensional (in fact one dimensional in our case).  We have already calculated the expressions on the left hand side of (\ref{eq:lmh}); the denominator is given by (\ref{denomform}) and the numerator is given by (\ref{numform}).  The result follows from inserting these formulas into (\ref{eq:lmh}). 
\end{proof}

\begin{remark}
{\em It would be desirable to have the correction value in formula (\ref{correctedeig}) be independent of $\epsilon$. 
Note that from Lemma \ref{bcbounds} we have that the term $\left<\theta_\epsilon^{(2)}, \phi\right>$ is bounded with respect to $\epsilon$, so any sequence of $\epsilon_k\rightarrow 0$ will have subsequential limits. We would like to characterize its limit points, or at least rewrite it in a more explicit form.  For similar eigenvalue problems in second order homogenization in bounded domains \cite{GeMa1, GeMa2, MoVo1, MoVo2, SaVo}, the precise value of the boundary corrector limit is complicated by two factors 
(i) The limit may not be unique if the domain has boundary with flat parts of rational or infinite slope. We expect that the limit will be unique for smooth $D$ in which the boundary has no flat parts. (ii)  Even when the limit is unique, there is no known explicit characterization of the limit.  It may very well be the case that the first order transmission eigenvalue corrections exhibit both of these complications.  One needs to study the behavior of the boundary correctors for fourth order homogenization problems, and this is the subject of future work.  In the next section we consider the one-dimensional case, which is easier to analyze, and demonstrates that the corrector is not generically zero. Furthermore, if the scatterer has flat parts with rational or infinite slope, the one dimensional study suggests that the corrector will depend on how the boundary cuts the microstructure.}
\end{remark}
\section{The one dimensional case.}
Although we explicitly took the dimension $d=2$ or $d=3$, the same results clearly hold for $d=1$. Let us take $D=(0,1)$ for simplicity, while noting that the following can easily be extended with small modifications general intervals $(a,b)$. 
Recalling that the eigenfunction has zero Cauchy data at the boundary, the one dimensional boundary corrector function $\theta_\epsilon^{(2)}$ here satisfies 
\begin{eqnarray}   \left({d^2\over{dx^2}}+\tau_0\right) \frac{1}{n(x/\epsilon)-1} \left( {d^2\over{dx^2}}+\tau_0\right) \theta_\epsilon^{(2)} +\tau_0^2\theta_\epsilon^{(2)} =  0 \ \ \ \mbox{on}\ \ (0,1) & &\label{theta1d} \\ 
\theta_\epsilon^{(2)}(0) = -\epsilon \frac{\beta(0)}{\tau_0(\overline{n}-1)} {d^2\phi\over{dx^2}}(0)  & & \nonumber \\ \theta_\epsilon^{(2)}(1) = -\epsilon \frac{\beta(1/\epsilon) }{\tau_0(\overline{n}-1)}{d^2\phi\over{dx^2}}(1)& & \nonumber \\ (\theta_\epsilon^{(2)})'(0) = -  \frac{1}{\tau_0(\overline{n}-1)}{d\beta\over{d y}}(0) {d^2\phi\over{dx^2}}(0)  -\epsilon   \frac{\beta(0)}{\tau_0(\overline{n}-1)} {d^3\phi\over{dx^3}} (0)  && \nonumber\\  (\theta_\epsilon^{(2)})'(1) = -  \frac{1}{\tau_0(\overline{n}-1)}{d\beta\over{d y}}(1/\epsilon) {d^2\phi\over{dx^2}}(1)  -\epsilon  \frac{\beta(1/\epsilon)}{\tau_0(\overline{n}-1)}  {d^3\phi\over{dx^3}}(1).\nonumber\end{eqnarray}
In the limit, the boundary terms with $\epsilon$ will disappear, and so the limit will be dominated by the first terms of the Neumann data. Notice that as $\epsilon\rightarrow 0$, this first term is fixed on the left but changing with $\epsilon$ on the right. We see here that the limit of this boundary data is not unique, and depends on the sequence $\epsilon_k \rightarrow 0$. Assume that $$\epsilon_k = {1\over{N_k+\delta}}$$
where $N_k\rightarrow\infty$ are integers, so that $1/\epsilon_k -\lfloor{1/\epsilon_k}\rfloor = \delta$, and 
$$ {d\beta\over{d y}}(1/\epsilon_k) = {d\beta\over{d y}}(\delta) $$
due to periodicity. The sequences $\epsilon_k$ for which the boundary corrector has a limit are those for which this cutoff $\delta_k$ has a limit $\delta$. We see that for fixed cutoff  the equation for the corrector becomes a standard fourth order homogenization problem. Hence the corrector converges in $L^2$ at order $\epsilon$ to $\theta^*$, the solution to 
\begin{eqnarray}   \left({d^2\over{dx^2}}+\tau_0\right) \frac{1}{\overline{n}-1} \left( {d^2\over{dx^2}}+\tau_0\right) \theta^{*} +\tau_0^2\theta^{*} &=&  0 \ \ \ \mbox{on}\ \ (0,1)\label{thetastar} \\ 
\theta^{*}(0) &=& 0  \nonumber \\ \theta^{*}(1) &=& 0  \nonumber \\ (\theta^{*})'(0) &=& -  \frac{1}{\tau_0(\overline{n}-1)}{d\beta\over{d y}}(0) {d^2\phi\over{dx^2}}(0) \nonumber \\  (\theta^{*})'(1) & = &-  \frac{1}{\tau_0(\overline{n}-1)}{d\beta\over{d y}}(\delta) {d^2\phi\over{dx^2}}(1) .\nonumber\end{eqnarray}
We therefore have an explicit formula for the transmission eigenvalue corrector in one dimension, summarized in the following Theorem. 
\begin{theorem}
 Assume the dimension $d=1$ with period cell $Y=[0,1]$, and $n_\epsilon:=n(x/\epsilon)\in L^\infty(D)$ is periodic in $y:=x/\epsilon$ such that $n(y)-1$ is positive uniformly in $Y$. Let $\tau_0$ be a simple transmission eigenvalue of the homogenized problem with constant refractive index $\bar n$, and $\phi$ the corresponding eigenfunction normalized such that $\|\phi\|_{L^2(D)}=1$.  Assume $$\epsilon_k = {1\over{N_k+\delta}}$$  where $N_k\rightarrow\infty$ are integers. Then for any $k$ sufficiently large, there exists a simple transmission eigenvalue $\tau_\epsilon>0$ of the periodic media with refractive index $n(x/\epsilon)$, which satisfies the following asymptotic expansion
 \begin{equation}\label{correctedeig}
 \tau_\epsilon=\tau_0+\epsilon\frac{\tau_0^2(1-\bar n)\left<\theta^{*}, \phi\right>_{L^2(D)}}{\bar n-3-2\tau_0\bar n \langle L^{-1}_{\tau_0,0}\phi,\phi\rangle_{L^2(D)}}+O(\epsilon^2)
 \end{equation} provided that $\langle L^{-1}_{\tau_0,0}\phi,\phi\rangle_{L^2(D)}\neq \frac{\bar n-3}{2\tau_0\bar n}$,  where $\theta^{*}$ is given by (\ref{thetastar}) and $L_{\tau_0,0}$ is given by (\ref{Ltau0}).
\end{theorem}
\section{Conclusions}
In this manuscript we derived an asymptotic expansion for the transmission eigenvalues of a scatterer with periodically varying index of refraction in the case when the contrast does not change sign. In this situation, we were able to use the fourth order formulation for the transmission eigenvalue problem, and its analysis required us to study a fourth order homogenization problem. The two scale asymptotics reveal a boundary corrector as the largest microstructure effect, and this boundary corrector appears in the correction formula for the transmission eigenvalues. It appears that this boundary corrector function may exhibit all of the difficulties of the boundary correctors in many second order homogenization problems; in particular the lack of an explicit formula for its limit for dimension $d\geq 2$, even in the case of smooth domains. 
An explicit formula would allow us to determine what information about the microstructure of the medium can be extracted from the transmission eigenvalues. 
The analysis of these fourth order boundary correctors is therefore of great interest, and is the subject of our future work. 

\section{Appendix: Technical Lemmas}
In this section, we will collect the technical lemmas that are necessary for  the results in this paper. 
\subsection{Derivative of $T_0(\tau)$ with respect to $\tau$}\label{derivative}
This Section is taken from \cite{CaMoRo} for the reader's convenience. To apply the theorem, we need to compute the derivative of $T_0(\tau)=\mathbb{A}^{-1}_{\tau,0}\mathbb{B}$ with respect to $\tau$, evaluated at a function $u$. However, since $\mathbb{B}$ does not depend on $\tau$, this problem is equivalent to the derivative of $\mathbb{A}^{-1}_{\tau,0}$ evaluated at $\mathbb{B}u$. Thus it is only necessary to compute ${\mathcal D}\mathbb{A}^{-1}_{\tau,0}$. With that in mind, for $u\in H^2_0(D)$ we define the solution map ${\mathcal L}_\tau$ to variational problem:
\begin{equation}\label{Lt}
\Delta\Delta \mathbb{A}_{\tau,0} {\mathcal L}_\tau u = \Delta \left ( \frac{1}{n_0-1} \mathbb{A}^{-1}_{\tau,0}u\right ) + \frac{1}{n_0-1} \Delta \mathbb{A}^{-1}_{\tau,0} u + 2\tau\left( \frac{1}{n_0-1}+1 \right )\mathbb{A}^{-1}_{\tau,0}u
\end{equation}
which exists and is bounded due to Riesz Representation. Further, define  for $u\in H^2_0(D)$,
\begin{equation}
u_{\tau}=\mathbb{A}^{-1}_{\tau,0}u.
\end{equation} 
Notice that by construction,
\begin{align}\label{derividen}
(\mathbb{A}_{\tau,0} {\mathcal L}_{\tau} u, \phi)_{H^2_0(D)} &=\left ( \frac{1}{n_0-1} u_\tau, \Delta \phi\right)_{L^2(D)} +\left ( \frac{1}{n_0-1} \Delta u_\tau, \phi\right)_{L^2(D)} \notag \\
&+ 2\tau\left ( \left( \frac{1}{n_0-1}+1 \right ) u_\tau, \phi\right )_{L^2(D)}
\end{align}
\begin{proposition} Let ${\mathcal L}$ be defined by  \label{LtP} and  $\tau>0$. Then the  derivative of $\mathbb{A}^{-1}_{\tau,0}$ with respect to $\tau$ is  $-{\mathcal L}_{\tau}$, that is, ${\mathcal D}\mathbb{A}^{-1}_{\tau,0}=-{\mathcal L}_\tau$.
\end{proposition}
\begin{proof}
Observe since $\mathbb{A}_{\tau+h,0}u_{\tau+h}=\mathbb{A}_{\tau,0}u_\tau=u$,
\begin{align}\label{preq1}
(\mathbb{A}_{\tau+h,0}(u_{\tau+h}-u_\tau+h {\mathcal L}_\tau u),\phi)=&
(\mathbb{A}_{\tau+h,0}u_{\tau+h}-\mathbb{A}_{\tau+h,0}u_\tau + h\mathbb{A}_{\tau+h,0}{\mathcal L}_\tau u, \phi)_{H^2_0(D)}\notag \\
=&(\mathbb{A}_{\tau,0}u_{\tau}-\mathbb{A}_{\tau+h,0}u_\tau + h\mathbb{A}_{\tau+h,0}{\mathcal L}_\tau u_\tau, \phi)_{H^2_0(D)}\notag \\
=&-(2th+h^2) \int_{D}\left( \frac{1}{n_0-1}+1 \right )u_\tau \phi \, \mathrm{d}x \notag \\
&-h
\int_D \frac{1}{n_0-1} (u_\tau \Delta\phi + \Delta u_\tau \phi)\,\mathrm{d}x\notag \\
&+ h(\mathbb{A}_{\tau+h,0}{\mathcal L}_\tau u, \phi)_{H^2_0(D)}.\end{align}
\noindent 
From the definition of the bilinear form, there exists a constant depending on $\tau$ and $D$ such that
\begin{align}
(\mathbb{A}_{\tau+h,0}&u,\phi)_{H^2_0(D)}=( \mathbb{A}_{\tau,0}u,\phi)_{H^2_0(D)} +h\left (\Delta u + \tau u, \frac{1}{n_0-1}\phi \right)_{L^2(D)} \notag \\ 
&+ h \left (\frac{1}{n_0-1} u, \Delta \phi + \tau \phi \right )_{L^2(D)} + 2h(\tau+h)\left ( \left( \frac{1}{n_0-1}+1 \right ) u,\phi \right )_{L^2(D)} \notag \\
&=  ( \mathbb{A}_{\tau,0}u,\phi)_{H^2_0(D)} + O(h\|u\|_{H^2_0(D)}\|\phi\|_{H^2_0(D)})
\end{align}
where the above estimate uses that $H^2_0$ is embedded in $C^0$.
Using the above inequality and (\ref{derividen}), we obtain
\begin{align}
(\mathbb{A}_{\tau+h,0}{\mathcal L}_\tau u, \phi)_{H^2_0(D)}=&
\left ( \frac{1}{n_0-1} u_\tau, \Delta \phi\right)_{L^2(D)} +\left ( \frac{1}{n_0-1} \Delta u_\tau, \phi\right)_{L^2(D)} \notag \\&
+ 2\tau\left ( \left( \frac{1}{n_0-1}+1 \right ) u_\tau, \phi\right )_{L^2(D)}
+O\left (h\|u_\tau\|_{H^2_0(D)}\|\phi\|_{H^2_0(D)}\right ).
\end{align}
Substituting this into (\ref{preq1}) yields
\begin{align}
(\mathbb{A}_{\tau+h,0}(u_{\tau+h}-u_\tau+h {\mathcal L}_\tau u),\phi)=&
-h^2 \int_{D} \left( \frac{1}{n_0-1}+1 \right )u_\tau \phi \mathrm{d} x + O(h^2\|\phi\|_{H^2_0(D)}) \notag \\
&\leq C h^2  \left( \frac{1}{n_0-1}+1 \right ) \|u_\tau\|_{L^2(D)} \|\phi\|_{H^2_0(D)}\notag \\
&+O\left (h^2\|u_\tau\|_{H^2_0(D)}\|\phi\|_{H^2_0(D)}\right ).
\end{align}
Of course, we have the bound
\begin{equation}
\|u_\tau\|_{L^2(D)} \leq C \| u_\tau \|_{H^2_0(D)} \leq C \| A^{-1}_{\tau,0}\|_{\mathcal{L}(H^2_0(D))} \|u \|_{H^2_0(D)}.
\end{equation}
Choosing $\phi= u_{\tau+h}-u_\tau +h {\mathcal L}_\tau u$, we have by coercivity that
\begin{equation}
C \|u_{\tau+h}-u_\tau +h {\mathcal L}_h u\|_{H^2_0(D)}=O\left (h^2\|u\|_{H^2_0(D)}\right )
\end{equation}
where $C$ can be chosen to be independent of $\tau$.
To finish, we divide by $h\|u\|_{H^2_0(D)}C$ and take the supremum over $u\in H^2_0(D)$,
\begin{equation}
\frac{
\| \mathbb{A}^{-1}_{\tau+h,0} - \mathbb{A}^{-1}_{\tau,0} + h {\mathcal L}_\tau \|_{\mathcal{L}(H^2_0(D))}}{h} = O(h).
\end{equation}
Therefore the Frechet derivative $D\mathbb{A}^{-1}_{\tau,0}(\tau) =  -{\mathcal L}_\tau$.
\end{proof}
\subsection{A trace theorem}\label{trace}
We define here the traces of the Cauchy data  of functions on $L_\Delta^2(D)$, where
 $$L^2_{\Delta}(D):=\left\{u\in L^2(D): \Delta u\in L^2(D)\right\}.$$
More precisely their trace and their normal derivative on the boundary live in $H^{-{\frac{1}{2}}}({\partial D})$ and $H^{-{\frac{3}{2}}}({\partial D})$, respectively.  Indeed if $u\in L^2_{\Delta}(D)$ then its trace  $u\in H^{-{\frac{1}{2}}}(\partial D)$ is defined  by duality using the identity
 $$\left<u,\tau\right>_{H^{-{\frac{1}{2}}}(\partial D), H^{{\frac{1}{2}}}(\partial D)}=\int_D\left(u\Delta w-w\Delta u\right)\,dx$$
 where $w\in H^2(D)$ is such that $w=0$ and $\partial w/\partial \nu=\tau$. Similarly, the trace of $\partial u/\partial \nu\in H^{-{\frac{3}{2}}}(\partial D)$ is defined  by duality using the identity
 $$\displaystyle{\left<\frac{\partial u}{\partial \nu},\tau\right>_{H^{-{\frac{3}{2}}}(\partial D), H^{{\frac{3}{2}}}(\partial D)}=-\int_D\left(u\Delta w-w\Delta u\right)\,dx}$$
 where $w\in H^2(D)$  is such that $w=\tau$ and $\partial w/\partial \nu=0$. The above shows that the trace operator
 $$u\in L_\Delta^2(D)\mapsto \left(u_{\partial D}, \frac{\partial u}{\partial \nu}|_{\partial D}\right)\in H^{-{\frac{1}{2}}}({\partial D})\times H^{-{\frac{3}{2}}}({\partial D})$$
 has continuous right inverse, i.e
 $$\|u\|_{H^{-{\frac{1}{2}}}({\partial D})}+\left\|\frac{\partial u}{\partial \nu}\right\|_{H^{-{\frac{3}{2}}}({\partial D})}\leq C\left(\|\Delta u\|_{L^2(D)}+\|u\|_{L^2(D)}\right)$$
 with $C>0$ independent of $u$.
\section{Acknowledgments} F. Cakoni was partially supported by NSF grant DMS-24-06313. S. Moskow was partially supported by NSF grants DMS-2008441 and DMS-2308200. The authors would also like to acknowledge the AIM institute and the SQuaRE program "Scattering Properties of Multiscale Heterogeneous Media", which also supported this research. 

\bibliographystyle{plain}
\bibliography{citationsR}
\end{document}